\documentclass[reqno,10pt,centertags]{amsart} 
\usepackage{amsmath,amsthm,amscd,amssymb,latexsym,esint,upref,stmaryrd,
enumerate,color,verbatim,yfonts,enumitem,comment}
\calclayout
\usepackage{hyperref} 
\newcommand*{\mailto}[1]{\href{mailto:#1}{\nolinkurl{#1}}}

\newcommand{\doi}[1]{\href{http://dx.doi.org/#1}{DOI:#1}}

\setenumerate{label=$($\textit{\roman*}$)$}

\newcommand{\bbN}{{\mathbb{N}}}

\newcommand{\bbR}{{\mathbb{R}}}

\renewcommand{\a}{\alpha}
\renewcommand{\b}{\beta}
\newcommand{\g}{\gamma}

\newcommand{\z}{\zeta}

\renewcommand{\Re}{\text{\rm Re}}
\renewcommand{\Im}{\text{\rm Im}}
\renewcommand{\ln}{\text{\rm ln}}

\newcommand{\norm}[1]{\lVert#1\rVert}

\newcommand{\lb}{\label}

\newcommand{\wti}{\widetilde}

\newcommand{\dott}{\,\cdot\,}

\newcommand{\bi}{\bibitem}

\newcommand{\Lt}{{L^2((a,b);dx)}}

\makeatletter
\def\theequation{\@arabic\c@equation}

\allowdisplaybreaks 
\numberwithin{equation}{section}

\newtheorem{theorem}{Theorem}[section]

\newtheorem{lemma}[theorem]{Lemma}
\newtheorem{corollary}[theorem]{Corollary}
\newtheorem{definition}[theorem]{Definition}

\newtheorem{example}[theorem]{Example}

\theoremstyle{remark}

\newenvironment{remark}[1][]{\refstepcounter{theorem}\par\medskip\noindent\textit{Remark~$\theexample. #1$} \rmfamily}{{\ }\hfill $\diamond$ \vspace{6pt}}

\begin{document}

\title[An elementary approach to integral inequalities]{An elementary approach to integral inequalities involving higher order derivatives}

\author[B.\ Rosenzweig]{Bart Rosenzweig}
\address[B.\ Rosenzweig]{Department of Mathematics, The Ohio State University \\
100 Math Tower, 231 West 18th Avenue, Columbus, OH 43210, USA}
\email{\mailto{rosenzweig.26@osu.edu}}

\author[J.\ Stanfill]{Jonathan Stanfill}
\address[J.\ Stanfill]{Division of Geodetic Science, School of Earth Sciences, The Ohio State University \\
275 Mendenhall Laboratory, 125 South Oval Mall, Columbus, OH 43210, USA}
\email{\mailto{stanfill.13@osu.edu}}

\date{\today}
\@namedef{subjclassname@2020}{\textup{2020} Mathematics Subject Classification}

\begin{abstract}
Motivated by previous work leveraging factorizations of second- and fourth-order differential operators, a general integral inequality involving higher order derivatives is proven by elementary means. It is then shown how this framework generalizes the notions of Hardy improving potentials and Bessel pairs. Numerous examples of inequalities both new and previously known in the literature are given that may be proven in this manner.
\end{abstract}

\maketitle

\section{Introduction}\lb{s1}

Inequalities which refine and generalize the classical Hardy's inequality
\begin{equation}\label{eq:hardy}
    \int_0^\infty |f'(x)|^2\,dx \geq \frac{1}{4}\int_0^\infty |x|^{-2}|f(x)|^2\,dx,\quad f\in C_0^\infty((0,\infty)),
\end{equation}    
and Rellich's inequality
\begin{equation}\label{eq:rellich}
    \int_0^\infty |f''(x)|^2\,dx \geq \frac{9}{16}\int_0^\infty |x|^{-4}|f(x)|^2\,dx,\quad f\in C_0^\infty((0,\infty)),
\end{equation}    
abound in the literature (cf. \cite{BEL15,GM13,KPS17} as well as the extensive references within \cite{GPPS23}). These two are the first in an infinite sequence of inequalities involving higher-order derivatives, all of which may be refined, meaning that the lower bound may be replaced by a strictly larger one, often a complicated combination of integrals of lower derivatives with respect to other weight functions. Such inequalities also hold in higher dimensions, and involve $L^p$ norms for $p \neq 2$, though in this paper we restrict our focus to the one-dimensional, $p=2$ case. Similar results to those proven here will be addressed in future work for higher dimensions and general $p$.

Numerous recent papers \cite{GL18,GNP23,GPPS23,GPS21,GPS24} prove refinements of \ref{eq:hardy} and \ref{eq:rellich} (including in higher dimensions) by factorizing an appropriate second- or fourth-order differential operator $S$ in the form $S=T^+T$, where $T$ is a first- or second-order operator and $T^+$ is its formal adjoint. Then by careful partial integration it is shown that the desired inequality is of the form $\langle f,T^+T f\rangle_{L^2} \geq 0$, and therefore manifestly true. In this work we simplify and systematize this approach, while also greatly extending it to  prove a generic integral inequality motivated by these previous studies. In particular, our main theorem is the following (see also Remark \ref{rem2.4} for relaxing the assumptions):

\begin{theorem}\label{t3.2}
Let $f\in C^\infty_0((a,b))$ $($infinitely differentiable continuous functions with compact support$)$, $-\infty\leq a<b\leq\infty$. Then for real-valued $a_m\in C((a,b))$ with $0\leq m\leq n$ 
and $ (a_j a_k)\in C^{k-j}((a,b))$ with $0\leq j< k\leq n,$ the following integral inequality holds$:$
\begin{equation}\label{3.2}
\int_a^b \big|a_n(x)f^{(n)} (x)\big|^2 \, dx\geq  \sum_{m=0}^{n-1}\int_a^b c_{n,m}(x)\big|f^{(m)}(x)\big|^2\, dx, 
\end{equation}
with, for $0\leq m\leq n-1$,
\begin{align}\label{3.6}
c_{n,m}(x)=-a_m^2(x)-\sum_{j=0}^m \sum_{k=\max\{j+1,2m-j\}}^n (-1)^{k-j} t_{k-j,m-j} &[a_j(x)a_k(x)]^{(k+j-2m)},
\end{align}
where $t_{n,k}$ are given explicitly by $t_{0,0}=2$, $t_{n,0} = 1,\ n \geq 1$, and for $1\leq k\leq \lfloor \tfrac{n}{2}\rfloor$,
\begin{equation}
\label{2.13}
t_{n,k}=(-1)^k \left[\binom{n-k}{k}+\binom{n-k-1}{k-1} \right]=(-1)^k \frac{n}{k}\binom{n-k-1}{k-1}.
\end{equation}
\end{theorem}

Since one has the trivial inequality
\begin{equation}\label{1.4}
\int_a^b \big|a_n(x)f^{(n)}(x)\big|^2\, dx\geq0,
\end{equation}
we obtain the following as an immediate corollary regarding refinements of the integral inequality \eqref{1.4}:

\begin{corollary}\label{t3.3}
The integral inequality \eqref{3.2} represents a refinement of \eqref{1.4} for all $f\in C_0^\infty((a,b))$ if  for all $x\in (a,b)$ and $0\leq m\leq n-1,\ n\in\bbN$, the following system of $n$ differential inequalities holds$:$
\begin{equation}\label{diffineq}
0\leq c_{n,m}(x)=-a_m^2(x)-\sum_{j=0}^m \sum_{k=\max\{j+1,2m-j\}}^n (-1)^{k-j} t_{k-j,m-j} [a_j(x)a_k(x)]^{(k+j-2m)}.
\end{equation}
\end{corollary}

In Section \ref{s2} we prove Theorem \ref{t3.2}, showing that it follows from an elementary (but as far as we are aware, absent in the literature) identity representing the product of a function and one of its derivatives as a linear combination of derivatives of squares of its intermediate derivatives (cf. Lemma \ref{l2.1}). Then in Section \ref{s3} we show that in the $n=1$ case, Theorem \ref{t3.2} may be used to reduce the proving of inequalities of the form \ref{3.2} to solving a linear ODE (essentially by noting that \ref{3.6} with the inequality replaced with equality is a Riccati equation, when regarded as an ODE for the unknown function $a_0$). This allows us to demonstrate that our framework contains the methods of Hardy improving potentials and Bessel pairs (common in the literature for proving refinements of inequalities, cf. \cite{GM11,GM13}) as a special case. Finally, in Section \ref{s4} we give a series of examples of inequalities proven using Theorem \ref{t3.2}, some of which have appeared in the literature before, others of which are apparently new. In particular, we illustrate the ease of optimizing \eqref{diffineq} as well as utilizing the methods developed in Section \ref{s3}.

Some comments on notation: Superscripts in parentheses refer to derivatives with respect to the variable $x$, whereas those without parentheses are powers; $C_0^\infty((a,b))$ is the space of infinitely differentiable continuous functions on $(a,b)$ with compact support; $\bbN$ represents the positive integers while $\bbN_0=\bbN\cup\{0\}$; $\binom{n}{k}=n!/(k!(n-k)!),\ n,k\in\bbN$ will denote the binomial coefficient.

\section{Main results}\lb{s2}

In this section, we state the starting point of our integral inequality investigations and prove some general derivative identities that will be needed to prove our main results on integral inequality in Theorem \ref{t3.2}.
We begin by considering the $n$th-order differential expression
\begin{equation}\label{2.1}
T=\sum_{k=0}^n a_k(x) \frac{d^k}{dx^k},
\end{equation}
where $a_k\in C^n((a,b)),\ k\in\{0,1,\dots,n\},\ -\infty\leq a<b\leq\infty$ to begin with (though these assumptions will be relaxed later). We further assume that $a_k$ are real-valued throughout. 
The formal adjoint, $T^+$, is then given by
\begin{equation}
T^+=\sum_{k=0}^n (-1)^k \frac{d^k}{dx^k} a_k(x).
\end{equation}
Note that one then has by the general Leibniz rule,
\begin{align}
\begin{split}\label{2.3}
(T^+ T f)(x)&=\sum_{k=0}^n (-1)^k \sum_{j=0}^n \sum_{m=j}^{j+k} {k \choose m} [a_k(x)a_j(x)]^{(k-m+j)} f^{(m)}(x).
\end{split}
\end{align}

In order to develop the theory of inequalities related to factorizations of the type given by the $2n$th-order differential expression $T^+T$, we study inequalities derived from the fact that
\begin{equation}\label{2.6}
0\leq\norm{Tf}^2_{\Lt}=\langle f,T^+Tf\rangle_\Lt,\quad f\in C_0^\infty((a,b)).
\end{equation}
We would like to point out that previous studies utilizing this method used the form of $T^+Tf$ and then integration by parts on $\langle f,T^+Tf\rangle_{L^2}$ as the natural starting point (cf. \cite{GL18,GPPS23,GPS24}). Considering domains of operators associated with the $2n$th-order differential expression $T^+T$ is quite natural as this leads to appropriate function spaces in which \eqref{2.6} holds. In particular, studying the domains of the minimal operator and Friedrichs extension in the self-adjoint setting often allows one not only to extend the function spaces in which the inequality \eqref{2.6} holds but also to prove the optimality of constants appearing in the inequality via spectral theory (cf. \cite[Sect. 5.1]{FS25}, \cite{GPS21}). However, from  \eqref{2.3}, one can see how it becomes quite cumbersome to gather (and cancel) like terms in higher order settings to arrive at an inequality containing only $|f^{(m)}|^2$ terms, motivating our alternative starting point. 

Returning to the consideration of \eqref{2.6}, one notes from \eqref{2.1} that it suffices to study the structure of
\begin{align}\label{2.7}
0\leq\int_a^b |(Tf)(x)|^2 \, dx=\int_a^b \bigg|\sum_{k=0}^n a_k(x) f^{(k)}(x)\bigg|^2\, dx.
\end{align}
The highest order derivative term of $\eqref{2.7}$ would have integrand $\big|a_n(x)f^{(n)}(x)\big|^2$ so that one immediately knows that for this term
\begin{equation}\label{2.9}
\int_a^b \big|a_n(x)f^{(n)}(x)\big|^2\, dx\geq0,
\end{equation}
while any inequality given by \eqref{2.7} would be of the form
\begin{equation}\label{2.10}
\int_a^b \big|a_n(x)f^{(n)}(x)\big|^2\, dx\geq \int_a^b \big|a_n(x)f^{(n)}(x)\big|^2-\bigg|\sum_{k=0}^n a_k(x) f^{(k)}(x)\bigg|^2\, dx.
\end{equation}
For this to represent a refinement of the trivial inequality \eqref{2.9}, one must study the sufficient conditions for the right-hand side of \eqref{2.10} to be positive.

To this end, we will show a convenient and, at first glance, somewhat surprising way one can write the integrand on the right-hand side of \eqref{2.10}. We begin by noting, for real-valued $f$, the integrand can be expressed as
\begin{align}\label{2.11}
\big[a_n(x)f^{(n)}(x)\big]^2-\bigg[\sum_{k=0}^n a_k(x) f^{(k)}(x)\bigg]^2&=-\sum_{k=0}^{n-1} \big[a_k(x)f^{(k)}(x)\big]^2 -2\sum_{\underset{j< k}{j,k=0}}^n a_j(x)a_k(x) f^{(j)}(x)f^{(k)}(x).
\end{align}
The first sum on the right-hand side of \eqref{2.11} is already in the form we need, as it only contains squares of $f$ and its derivatives, without cross-terms. The second sum can be brought into a similar form by applying an appropriate identity. Though useful and intriguing in its own right, we have not found this particular identity in the literature. Therefore, we provide two proofs and some remarks, as the identity seems to us to be of independent interest.

\begin{lemma}\label{l2.1}
For $f\in C^n((a,b)),\ n\in\bbN,\ -\infty\leq a<b\leq\infty$, and $t_{n,k}$ as in Theorem \ref{t3.2}, the following identity holds$:$
\begin{equation}\label{2.12}
2f(x)f^{(n)}(x)=\sum_{k=0}^{\lfloor \frac{n}{2}\rfloor} t_{n,k} \big\{\big[f^{(k)}(x)\big]^2\big\}^{(n-2k)},\quad x\in(a,b).
\end{equation}
\end{lemma}
\begin{proof}
Direct calculation readily shows that \eqref{2.12} holds for $n=0,1,2$.
We now proceed by (strong) induction by assuming \eqref{2.12} is true for $n$ and $n-1$. Then
\begin{align}
2f(x)f^{(n+1)}(x)&=2\big[f(x)f^{(n)}(x)]'-2f'(x)f^{(n)}(x)=2\big[f(x)f^{(n)}(x)]'-2f'(x)[f'(x)]^{(n-1)} \notag\\
&=\sum_{k=0}^{\lfloor \frac{n}{2}\rfloor} t_{n,k} \big\{\big[f^{(k)}(x)\big]^2\big\}^{(n+1-2k)}-\sum_{k=0}^{\lfloor \frac{n-1}{2}\rfloor} t_{n-1,k} \big\{\big[f^{(k+1)}(x)\big]^2\big\}^{(n-1-2k)}  \notag\\
&=\big\{[f(x)]^2\big\}^{(n+1)}+\sum_{k=1}^{\lfloor \frac{n+1}{2}\rfloor} (t_{n,k} - t_{n-1,k-1}) \big\{\big[f^{(k)}(x)\big]^2\big\}^{(n+1-2k)},
\end{align}
where we have used the fact that $t_{n,k}=0$ for $k>\lfloor\tfrac{n}{2}\rfloor,$ $n\geq2$.
Hence, it suffices to verify that the coefficients defined in \eqref{2.13} satisfy the recursion $t_{n+1,k}=t_{n,k}-t_{n-1,k-1},\ 1\leq k\leq \lfloor \tfrac{n+1}{2}\rfloor,$
which follows from direct calculation since
\begin{align}
t_{n,k}-t_{n-1,k-1}&=(-1)^k \left[\binom{n-k}{k}+\binom{n-k-1}{k-1} +\binom{n-k}{k-1}+\binom{n-k-1}{k-2} \right] \notag\\
&=(-1)^k \left[\binom{n+1-k}{k}+\binom{n+1-k-1}{k-1} \right] =t_{n+1,k},
\end{align}
completing the proof.
\end{proof}

\begin{remark} 
The coefficients $t_{n,k}$ given in Lemma \ref{l2.1} coincide with the integer sequence \cite[A213234]{OEIS} (see also \cite[A034807]{OEIS}) which gives the coefficient of $\xi^{n-2k}$ in the generalized Lucas polynomial $V_n(\xi,-1)$ (cf. \cite{BH74}). In particular, the coefficient of the $(n-2k)$-th derivative on the right-hand side of \eqref{2.12} agrees with the coefficient of the $(n-2k)$-th power appearing in $V_n(\xi,-1)$. This appears to be a new relationship between this integer sequence and differentiation. Thus, further study of \eqref{2.12}, perhaps relating differentiation to other known interpretations of $t_{n,k}$, would be of natural interest.
\end{remark}

Though the above proof of Lemma \ref{l2.1} is the quickest proof we know, the seemingly most direct proof would be via the general Leibniz rule for $n$-times differentiable functions $f,g$:
\begin{equation}\label{2.20}
(f(x)g(x))^{(n)}=\sum_{k=0}^n \binom{n}{k}f^{(n-k)}(x)g^{(k)}(x).
\end{equation}
However, a direct proof by using the general Leibniz rule on the right-hand side of \eqref{2.12} and applying elementary binomial coefficient identities becomes surprisingly cumbersome. We offer an alternative direct proof of Lemma \ref{l2.1} below by applying a related result, kindly shared with the authors by Daniel Herden in private correspondence \cite{H24}, which is proven via the general Leibniz rule.

\begin{lemma}[\cite{H24}]\label{l2.4}
For $f\in C^n((a,b)),\ n\in\bbN,\ -\infty\leq a<b\leq\infty$, the following identity holds for $x\in(a,b)$$:$
\begin{equation}\label{2.21}
\sum_{k=0}^{\lfloor \frac{n}{2}\rfloor} (-1)^k \binom{n-k}{k} \big\{\big[f^{(k)}(x)\big]^2\big\}^{(n-2k)}=\sum_{\ell=0}^n f^{(n-\ell)}(x)f^{(\ell)}(x).
\end{equation}
\end{lemma}
\begin{proof}

Applying the general Leibniz rule to the left-hand side of \eqref{2.21} and comparing coefficients on each side shows that it suffices to show
\begin{equation}\label{2.22}
\sum_{k=0}^{\lfloor \frac{n}{2}\rfloor} (-1)^k \binom{n-k}{k}\binom{n-2k}{\ell-k}=1,\quad \ell=0,1,\dots,n.
\end{equation}
One can rewrite \eqref{2.22} as
\begin{equation}\
\sum_{k=0}^{\lfloor \frac{n}{2}\rfloor} (-1)^k \binom{\ell}{k} \binom{n-k}{\ell}=1,\quad \ell=0,1,\dots,n,
\end{equation}
which is true by comparing the $x^{n-\ell}$ powers in the following equation where we use binomial series to expand each term:
\begin{equation}
\bigg(\sum_{k=0}^{\infty}(-1)^k\binom{\ell}{k}x^k\bigg)\bigg(\sum_{j=0}^\infty\binom{\ell+j}{\ell}x^j\bigg)=\frac{(1-x)^{\ell}}{(1-x)^{\ell+1}}=\dfrac{1}{1-x}=\sum_{m=0}^\infty x^m.
\end{equation}
\end{proof}

\begin{proof}[Alternate proof of Lemma \ref{l2.1}]
Lemma \ref{l2.1} follows from Lemma \ref{l2.4} by considering the difference
\begin{equation}
2f(x)f^{(n)}(x)=\sum_{\ell=0}^n f^{(n-\ell)}(x)f^{(\ell)}(x)-\sum_{\ell=1}^{n-1} f^{(n-\ell)}(x)f^{(\ell)}(x),
\end{equation}
and then applying \eqref{2.21} and regrouping. Therefore, this method along with the general Leibniz rule proof of Lemma \ref{l2.4} yields a proof relying only on the general Leibniz rule and binomial identities.
\end{proof}

We are now in a position to prove Theorem \ref{t3.2} by applying the above identities.

\begin{proof}[Proof of Theorem \ref{t3.2}]
We first point out that for $f = f_1 + i f_2 \in C_0^\infty((a,b))$, $(a,b)\subseteq\bbR$, with $f_1 = \Re(f)$, $f_2 = \Im(f)$, one has the obvious equality
\begin{equation}
\big|f^{(n)}\big|^2 = \big[f_1^{(n)}\big]^2 + \big[f_2^{(n)}\big]^2,
\end{equation}
and thus the inequality considered holds for all 
$f \in C^{\infty}_0((a,b))$ if and only if it holds for all real-valued 
$f \in C^{\infty}_0((a,b))$. Therefore, without loss of generality, we assume that $f\in C^\infty_0((a,b))$ is real-valued throughout the proof. 

Replacing $f$ by $f^{(j)}$ and substituting the result from Lemma \ref{l2.1} into \eqref{2.11} (noting that the difference in the number of derivatives, there $n$, is here $k-j$) yields
\begin{align}\label{3.1}
\big[a_n&(x)f^{(n)}(x)\big]^2-\bigg[\sum_{k=0}^n a_k(x) f^{(k)}(x)\bigg]^2\\
&=-\sum_{k=0}^{n-1} \big[a_k(x)f^{(k)}(x)\big]^2-\sum_{\underset{j< k}{j,k=0}}^n a_j(x)a_k(x) \sum_{\ell=0}^{\lfloor \frac{k-j}{2}\rfloor} t_{k-j,\ell} \big\{\big[f^{(j+\ell)}(x)\big]^2\big\}^{(k-j-2\ell)}.\notag
\end{align}

Directly substituting \eqref{3.1} into \eqref{2.10} and applying integration by parts ($k-j-2\ell$ times on each term, observing the support properties of $f$) yields
\begin{align}\label{2.21a}
\int_a^b &\big[a_n(x)f^{(n)}(x)\big]^2\, dx\\
&\geq -\sum_{k=0}^{n-1}\int_a^b \big[a_k(x)f^{(k)}(x)\big]^2\, dx+\sum_{\underset{j< k}{j,k=0}}^n \sum_{\ell=0}^{\lfloor \frac{k-j}{2}\rfloor} t_{k-j,\ell}\int_a^b a_j(x)a_k(x) \big\{\big[f^{(j+\ell)}(x)\big]^2\big\}^{(k-j-2\ell)}\, dx \notag\\
&=-\sum_{m=0}^{n-1} \int_a^b\big[a_m(x)f^{(m)}(x)\big]^2\, dx 
 +\sum_{\underset{j< k}{j,k=0}}^n  (-1)^{k-j} \sum_{\ell=0}^{\lfloor \frac{k-j}{2}\rfloor} t_{k-j,\ell} \int_a^b[a_j(x)a_k(x)]^{(k-j-2\ell)}\big[f^{(j+\ell)}(x)\big]^2\, dx.  \notag
\end{align}
Next, after an appropriate grouping of terms, in particular, by considering when $j+\ell=m\in\{0,1,\dots,n-1\}$ (and momentarily ignoring the upper limit in the sum over $\ell$ as including more terms only adds zero), one concludes that
\begin{align}\label{3.5}
\int_a^b \big[a_n(x)f^{(n)} (x)\big]^2 \, dx&\geq -\sum_{m=0}^{n-1}\int_a^b \bigg(a_m^2(x) \\
&\quad \, +\sum_{j=0}^m \sum_{k=j+1}^n (-1)^{k-j} t_{k-j,m-j} [a_j(x)a_k(x)]^{(k+j-2m)}\bigg)\big[f^{(m)}(x)\big]^2\, dx. \notag
\end{align}
Finally, noting that in the last sum in \eqref{3.5}, the coefficients $t_{k-j,m-j}$ are zero if $m-j>\tfrac{k-j}{2}$ or $k<2m-j$, yields \eqref{3.2} as the first nonzero $k$ term will be at $\max\{j+1,2m-j\}$.

The assumptions on $a_m$ and the products $(a_j a_k)$ follow immediately from inspection of \eqref{3.2}.
\end{proof}

Extending the proof of the inequality to more general function spaces where $f$ is not compactly supported can be done by considering the integration by parts step in \eqref{2.21a} (or domain considerations of operators associated with $T^+T$ as previously mentioned). In particular, it suffices to consider functions $f$ such that the boundary terms are zero.

\begin{remark}\label{rem2.4}
We note that the integral inequality \eqref{t3.2} does not require that $a_m\in C((a,b)),\ 0\leq m\leq n,$ or that the product of coefficient functions $(a_j a_k)$ belongs to $C^{k-j}((a,b)),\ 0\leq j<k\leq n$, since the Leibniz rule holds for weak derivatives. In order for the integrals in \eqref{3.5} to be defined, it is enough to require the following:
    \begin{enumerate}
        \item[$(a)$] For each $0 \leq k \leq n$, $a_k \in L^2_{\text{loc}}((a,b))$.
        \item[$(b)$] For each $0\leq j<k \leq n$, the product $a_ja_k$ has weak derivatives of order $k-j-2\ell$ for all $\ell = 0,1,\dots,\lfloor \frac{j+k}{2}\rfloor$.
    \end{enumerate}
    For example, if $n = 1$, we require $a_0,a_1\in L^2_{\text{loc}}((a,b))$ and $(a_0a_1)' \in L^1_{\text{loc}}((a,b))$. 
\end{remark}

\section{Coefficients of main integral inequality}\lb{s3}

In this section, we further investigate the coefficients for the integral inequality given in Theorem \ref{t3.2} (explicitly written in \eqref{3.6}). We also relate our differential inequalities in the first-order case to the notions of Hardy improving (HI) potentials and Bessel pairs.

We begin by studying the coefficient of the integral inequality \eqref{3.2} when $n=1$, namely
\begin{equation}\label{4.1}
c_{1,0}(x)=-a_0^2(x)+[a_0(x)a_1(x)]'.
\end{equation}

As a starting point, suppose one wants to investigate a first-order inequality of the form \eqref{3.2} for some choice of the highest order coefficient, $a_1(x)$, and of the lowest-order one $c_{1,0}(x)=g(x)$ where $g(x) \in L^1_{loc}((a,b))$, $g(x)\geq0$. Then one can set \eqref{4.1} equal to $g(x)$ and solve for the unknown function $a_0(x)$, which becomes equivalent to studying the Riccati equation (assuming $a_1(x)\neq0$ on $(a,b)$)
\begin{equation}\label{4.2}
a_0'(x)=\frac{1}{a_1(x)}a_0^2(x)-\frac{a_1'(x)}{a_1(x)}a_0(x)+\frac{g(x)}{a_1(x)}.
\end{equation}
To transform \eqref{4.2} into a second-order linear equation, we multiply both sides by $1/a_1(x)$ (assuming $1/a_1(x)\neq0$ on $(a,b)$) to rewrite \eqref{4.2} (after subtracting appropriate terms on both sides)
\begin{equation}\label{4.3}
\left(\frac{a_0(x)}{a_1(x)}\right)'=\left(\frac{a_0(x)}{a_1(x)}\right)^2-2\frac{a_1'(x)}{a_1(x)}\left(\frac{a_0(x)}{a_1(x)}\right)+\frac{g(x)}{a_1^2(x)}.
\end{equation}
Next, substituting $a_0(x)/a_1(x)=-u'(x)/u(x)$, with $u(x)\neq 0$ on $(a,b)$, into \eqref{4.3} and rewriting yields the second-order linear differential equation
\begin{equation}\label{4.4}
u''(x)+2\frac{a_1'(x)}{a_1(x)}u'(x)+\frac{g(x)}{a_1^2(x)}u(x)=0.
\end{equation}
Finally, multiplying \eqref{4.4} by $-a_1^2(x)$ yields the Sturm--Liouville equation
\begin{equation}\label{4.5}
-\big(a_1^2(x) u'(x)\big)'-g(x)u(x)=0,\quad x\in(a,b),
\end{equation}
where we note that a solution, $u$, to \eqref{4.5} will yield a solution to \eqref{4.2} given by
\begin{equation}\label{4.6}
a_0(x)=-a_1(x)\frac{u'(x)}{u(x)}.
\end{equation}

We now make a few comments that will be used in our examples. First, by \eqref{4.6}, we note that any constants multiplying $u(x)$ will be canceled out when writing $a_0$. Second, we have assumed that $u(x)\neq0$ on $(a,b)$ in the previous steps which will naturally be manifest in the interval that the underlying inequality will be valid on (and directly connect to the notion of HI potentials). These two observations lead to considering the existence of positive solutions to \eqref{4.5}, that is, oscillation theory for the equation \eqref{4.5}, for which we refer to \cite[Ch. 8]{GNZ24} and the references therein.

We also note that the condition \eqref{4.5} for validity of the integral inequality \eqref{3.2} for $n = 1$ was previously obtained in \cite[Thm. 3.1]{B71} in a slightly more general setting.

\begin{remark}
It is tempting to extend the previous considerations to the coefficients of the integral inequality \eqref{3.2} when $n=2$ given by
\begin{align}\label{4.50}
c_{2,0}(x)&=-a_0^2(x)+[a_0(x)a_1(x)]'-[a_0(x)a_2(x)]'',\quad
c_{2,1}(x)=-a_1^2(x)+2a_0(x)a_2(x)+[a_1(x)a_2(x)]'.
\end{align}
In particular, a very natural way to proceed is for one to treat $a_0(x)$ and $a_1(x)$ as unknown functions in \eqref{4.50} by supposing the highest order coefficient, $a_2(x)$, has been chosen along with $c_{2,0}(x)=g_0(x)$ and $c_{2,1}(x)=g_1(x)$ for some $L^1_{loc}$ functions $g_j(x)\geq0,\ j=0,1$, on $(a,b)$. However, there does not seem to be a canonical way to study this system of two nonlinear equations as in the first-order case. One might try to solve for $a_0(x)$ in the second line of \eqref{4.50} (supposing $a_2(x)\neq0$),
and substitute the resulting expression into the first line, arriving at a third-order nonlinear differential equation for $a_1(x)$. However, this does not seem to simplify the analysis in any practical way.
\end{remark}

\subsection{Relation to Hardy improving potentials and Bessel pairs}\label{s4.1.1}

We now show our method developed here for $n=1$ is an extension of the related notions of Hardy improving potentials and Bessel pairs that were developed in \cite{GM11} and \cite{GM13}.

\begin{definition}[{\cite[Def. 1.1.1]{GM13}}]\label{d4.2a}
We say that a nonnegative real-valued $C^1$ function $P$ is a Hardy improving potential $($HI potential$)$ on $(0,R)$ if there exists $c>0$ such that the following equation has a positive solution on $(0,R)$$:$
\begin{equation}\label{HI}
y''(r)+r^{-1} y'(r)+c P(r) y(r)=0.
\end{equation}
\end{definition}

Note that this notion of Hardy improving potentials is essentially asking when one can improve Hardy's inequality by adding a potential to the underlying Bessel equation that yields the inequality. 

To compare to our method, note that an integrating factor transforms \eqref{HI} to
\begin{equation}
-\big(r y'(r)\big)'-cr P(r) y(r)=0,
\end{equation}
which corresponds to choosing $a_1^2(x)=x$ and $g(x)=cx P(x)$ in \eqref{4.5}. 

Alternatively, one can connect to our method via a Liouville transform by first setting $a_1(x)=1$ and $g(x)=(2x)^{-2}+h(x)$ for some function $h(x)\geq 0$ on $(0,R)$ in \eqref{4.5} to write
\begin{equation}\label{HI2}
-u''(x)-\big((2x)^{-2}+h(x)\big) u(x)=0,\quad x\in(0,R).
\end{equation}
To compare \eqref{HI} and \eqref{HI2}, one applies a Liouville transform as in \cite[Sect. 7]{Ev05} with $k=1,\ K=1$ (see also \cite[Sect. 4]{FPS25}, \cite[Thm. 3.5.1]{GNZ24}) to \eqref{HI}. Direct computation then yields $x(r)=r$ for $r\in(0,R)$ and the new potential function will be $Q(x)=-cP(x)-(2x)^{-2},$
with the solution to the new problem given by $x^{1/2}y(x)$.
Therefore, the equations \eqref{HI} and \eqref{HI2} are equivalent upon identifying $h(x)=cP(r)$ with $x=r$ and $u(x)=x^{1/2}y(x)$ on $(0,R)$.

Therefore, our method developed above is a quite natural extension of HI potentials. Choosing $a_1$ and some $\wti{g}\geq0$ (nonnegativity only assumed to give a refinement but not needed) to give a known inequality of the form
\begin{equation}
\int_a^b a_1^2(x) |f'(x)|^2\, dx \geq \int_a^b \wti{g}(x) |f(x)|^2\, dx,\quad f\in C^\infty_0((a,b)),
\end{equation}
studying solutions to the equation \eqref{4.5} with $g(x)=\wti{g}(x)+h(x)$ for $h(x)\geq0$ will determine the existence of improving inequalities to the original known one. For other results regarding improving Hardy-type inequalities via perturbations of the potential we refer to \cite{GNP23}.

We now turn to comparing our method to the notion of Bessel pairs.

\begin{definition}[{\cite[Sect. 1.3]{GM13}}]
We say that a pair of $C^1$ functions $(V,W)$ is a $k$-dimensional Bessel pair on $(0,R)$ if there exists $c>0$ such that the following equation has a positive solution on $(0,R)$$:$
\begin{equation}\label{3.12a}
y''(r)+\left(\frac{k-1}{r}+\frac{V'(r)}{V(r)}\right) y'(r)+\frac{c W(r)}{V(r)} y(r)=0.
\end{equation}
\end{definition}

Once again, using an integrating factor transforms \eqref{3.12a} into
\begin{equation}
-\big(r^{k-1} V(r) y'(r)\big)'-cr^{k-1} W(r) y(r)=0,
\end{equation}
which corresponds to choosing $a_1^2(x)=x^{k-1} V(x)$ and $g(x)=cx^{k-1} W(x)$ in \eqref{4.5}.

Alternatively, a similar comparison as with HI potentials using a Liouville transform can be performed under additional assumptions on $V$, namely $V>0$ and $V'\in AC_{loc}((a,b))$.

\subsection{Alternative method}
As an alternative way to investigate the coefficient $c_{1,0}(x)$, we note that one could also consider \eqref{4.1} by choosing $c_{1,0}(x)=g(x)\geq0$ on $(a,b)$ again but with $a_1(x)$ as the unknown function, rather than $a_0(x)$. In this case, one simply has the first-order linear differential equation
\begin{equation}
[a_0(x)a_1(x)]'=a_0^2(x)+g(x),
\end{equation}
which, assuming $a_0(x)\neq 0$ on $(a,b)$, has the solution
\begin{equation}
a_1(x)=\frac{1}{a_0(x)}\int^x a_0^2(t)+g(t)\, dt.
\end{equation}
This leads to the family of inequalities
\begin{equation}
\int_a^b \left[\frac{1}{a_0(x)}\int^x [a_0^2(t)+g(t)]\, dt\right]^2 |f'(x)|^2\, dx \geq \int_a^b g(x)|f(x)|^2\, dx,\quad f\in C^\infty_0((a,b)).
\end{equation}

\section{Examples}\lb{s4}

In this section, we apply our previous results to recover known inequalities and refinements as well as prove new inequalities. We note that these examples are meant to be instructive for applying these methods in the $n=1$ and $n=2$ cases, particularly illustrating the analysis of Section \ref{s3} when $n=1$.

\subsection{First-order setting}\lb{s5.1}

Specializing to the case $n=1$, from \eqref{3.2} and \eqref{diffineq} one sees that the integral inequality and associated differential inequality that would need to be satisfied for a refinement are
\begin{align}\label{5.1}
\int_a^b a_1^2(x) |f'(x)|^2\, dx &\geq \int_a^b \big\{-a_0^2(x)+[a_0(x)a_1(x)]'\big\}|f(x)|^2\, dx,\quad f\in C^\infty_0((a,b)),  \notag  
 \\
c_{1,0}(x)&=-a_0^2(x)+\big[a_0(x)a_1(x)\big]'\geq0,\ \text{on}\ (a,b).
\end{align}
We illustrate choosing $a_0$, $a_1$, and optimizing $c_{1,0}(x)\geq0$ as well as the method outlined in Section \ref{s3}.

\begin{example}[Power weighted Hardy's inequality] 
Choosing $a_0(x)=[(\g-1)/2]x^{(\g/2)-1},$ $a_1(x)=x^{\g/2},$ $\g\in\bbR$, on $(0,\infty)$ yields the power weighted Hardy's inequality $($with optimal constant$)$
\begin{equation}\label{eq:weighthardy}
\int_0^\infty x^\gamma |f'(x)|^2\, dx\geq \frac{(1-\gamma)^2}{4}\int_0^\infty x^{\gamma-2}|f(x)|^2\, dx,\ \ \ \g\in\bbR,\ f\in C^\infty_0((0,\infty)).
\end{equation}
To motivate these choices, choosing $a_1(x)=x^{\g/2}$ and $a_0(x)=\a x^{(\g/2)-1},$ $\a,\g\in\bbR$ requires by \eqref{5.1} that $\a(\g-1) x^{\g-2}-\a^2 x^{\g-2}\geq0$ or $\a(\a+1-\g)\leq0$. This is minimized at $\a=(\g-1)/2$, yielding \eqref{eq:weighthardy}.
\end{example}

\subsubsection{Refinements of power weighted Hardy's inequality}

We now illustrate the methods of Section \ref{s3}.

\begin{example}
Note that the method used above to prove \eqref{eq:weighthardy} does not account for variations in the interval, yielding $(1-\g)^2/4$ for all $(a,b)\subseteq(0,\infty)$ with no dependence on the endpoints. To probe this dependence we can use the method outlined in Section \ref{s3} by considering the Sturm--Liouville problem
\begin{equation}
-\big(x^\g u'(x)\big)'-\left(\frac{(1-\g)^2}{4}+c \right)x^{\g-2} u(x)=0,\quad \g\in\bbR,\ c>0,\ x\in(a,b),\ 0<a<b<\infty.
\end{equation}
A solution to this equation is given by $u(x)=x^{(1-\g)/2} \sin\big(\sqrt{c}\, \ln(x/a)\big)$ which is zero at $x=a$ and $c=[\pi/\ln(b/a)]^2$ is the largest choice of $c$ such that $u(x)$ is positive on $(a,b)$. Therefore, choosing $a_1(x)=x^{\g/2}$ and $a_0(x)=-x^{\g/2}u'(x)/u(x)$ with $u(x)=x^{(1-\g)/2} \sin\big(\pi \ln(x/a)\ln(b/a)\big)$ yields the inequality $($with the constant optimal as recently proven in \cite{GP25}$)$
\begin{equation}
\int_a^b x^\gamma |f'(x)|^2\, dx\geq \left(\frac{(1-\gamma)^2}{4}+\frac{\pi^2}{[\ln(b/a)]^2}\right)\int_a^b x^{\gamma-2}|f(x)|^2\, dx,\ \ \ \g\in\bbR,\ f\in C^\infty_0((a,b)),
\end{equation}
when $0<a<b<\infty$.
Notice that the constant becomes $(1-\g)^2/4$ whenever $a\downarrow0$ and/or $b\uparrow\infty$.
\end{example}

\begin{example}[Power weighted Hardy's inequality with log refinement]\label{ex4.2} 
For this example, let $a_1(x)=x^{\g/2}$ and consider the Sturm--Liouville problem, for $\g\in\bbR,\ \eta\in[e_N R,\infty),\ N\in\bbN,\ R\in(0,\infty)$,
\begin{align}\label{4.17a}
-\big(x^\g u'(x)\big)'-\frac{1}{4}x^{\g-2}\Bigg((1-\g)^2+\sum_{k=1}^{N} \prod_{p=1}^{k} [\ln_{p}(\eta/x)]^{-2} \Bigg)u(x)=0,\quad x\in(0,R),
\end{align}
where the iterated logarithms $\ln_k( \, \cdot \, )$ and exponentials $e_k$, $k \in \bbN$, are given by
\begin{equation}\label{4.18a}
\ln_1(\, \cdot \,) = \ln( \, \cdot \, ), \quad \ln_{k+1}( \, \cdot \,) = \ln \big( \ln_k(\, \cdot \,) \big), \quad e_1 =1, \quad e_{j+1} = e^{e_k},\quad k \in \bbN.
\end{equation}
A solution to \eqref{4.17a} is given by the function
\begin{equation}
u_{\g,\eta,N}(x)=x^{(1-\g)/2}\prod_{p=1}^{N} [\ln_{p}(\eta/x)]^{1/2},
\end{equation}
so we further choose $a_0(x)=-x^{\g/2}\big(u'_{\g,\eta,N}(x)/u_{\g,\eta,N}(x)\big)$ to arrive at the inequality \cite[Lemma 2.1]{GPS24}
\begin{align}
\int_0^R x^\g |f'(x)|^2\, dx&\geq \frac{(1-\g)^2}{4} \int_0^R x^{\g-2} |f(x)|^2\, dx +4^{-1}\int_0^R x^{\g-2} \Bigg(\sum_{k=1}^{N} \prod_{p=1}^{k} [\ln_{p}(\eta/x)]^{-2} \Bigg)|f(x)|^2\, dx, \notag\\
&\hspace{1cm} \g\in\bbR,\ \eta\in[e_N R,\infty),\ N\in\bbN,\ R\in(0,\infty),\quad f\in C_0^\infty ((0,R)).
\end{align}
\end{example}

\begin{example}[Power weighted Hardy's inequality with Bessel zero refinement]\label{ex4.1a} 
We prove the following integral inequality which is an extension of several known inequalities $($see \eqref{eq:weighthardy} and \cite[Eq. (3.48)]{GPS21}$)$$:$ 
\begin{align}\label{4.7a}
\int_0^b x^\g |f'(x)|^2\, dx & \geq \frac{(1-\g)^2-(2+\mu-\g)^2\nu^2}{4}\int_0^b x^{\g-2}|f(x)|^2\, dx +\frac{j_{\nu,1}^2(2+\mu-\g)^2}{4b^{2+\mu-\g}}\int_0^b x^\mu |f(x)|^2\, dx,\notag\\
&\hspace{3cm} f\in C_0^\infty((0,b)),\ \nu\in[0,\infty),\ 2+\mu-\g>0,\ b\in(0,\infty),
\end{align}
where the first coefficient on the right-hand side of \eqref{4.7a} is positive for
\begin{equation}
0\leq\nu< \frac{|1-\g|}{2+\mu-\g},\quad 2+\mu-\g>0,
\end{equation}
and $j_{\nu,k}$ denotes the $k$th positive zero of the Bessel function $J_\nu(\dott)$. Notice that the form of \eqref{4.7a} allows one to adjust the coefficients on each integral to more heavily weight one over the other.

The choices of coefficients that lead to this from factorization are not at all obvious and given by
\begin{equation}\label{4.9a}
a_1(x)=x^{\g/2},\quad a_0(x)=\frac{\g-1}{2}x^{(\g/2)-1}-x^{\g/2}\frac{\big[J_\nu\big(j_{\nu,1}(x/b)^{(2+\mu-\g)/2}\big)\big]'}{J_\nu\big(j_{\nu,1}(x/b)^{(2+\mu-\g)/2}\big)},
\end{equation}
with $\nu\in[0,\infty)$ and $2+\mu-\g>0$. From \eqref{4.1}, one can readily verify that these choices do indeed yield the inequality \eqref{4.7a}, with the methods of Section \ref{s3} the motivation for these choices.

Beginning with the power weighted Hardy's inequality on $(0,\infty)$, we choose $a_1(x)=x^{\g/2}$ and investigate a perturbation of the underlying potential for Hardy on $(0,b)$ with $b\in(0,\infty)$ of the form
\begin{equation}
c_{1,0}(x)=g(x)=\frac{(1-\g)^2-(2+\mu-\g)^2\nu^2}{4}x^{\g-2}+c x^\mu,\quad \nu\in[0,\infty),\ 2+\mu-\g>0,\ c>0.
\end{equation}
Then from \eqref{4.5} we consider the Sturm--Liouville problem, for $\nu\in[0,\infty),\ 2+\mu-\g>0$,
\begin{align}\label{4.11a}
-\big(x^\g u'(x)\big)'-\left(\frac{(1-\g)^2-(2+\mu-\g)^2\nu^2}{4}x^{\g-2}+c x^\mu\right)u(x)=0,\quad x\in(0,b).
\end{align}
We remark that the parameter choices here are to ensure that the differential equation is nonoscillatory.
Two linearly independent solutions to \eqref{4.11a} are given by
\begin{align}\label{4.12a}
u_1(x)&=x^{(1-\g)/2} J_\nu\big(2c^{1/2}x^{(2+\mu-\g)/2}/(2+\mu-\g)\big),\\
u_2(x)&=x^{(1-\g)/2} Y_\nu\big(2c^{1/2}x^{(2+\mu-\g)/2}/(2+\mu-\g)\big).
\end{align}
Recalling equation \eqref{4.6} for $a_0$, one must find the largest first root of any linear combination of $u_1$ and $u_2$. A straightforward analysis using properties of Bessel functions zeros $($see, e.g., \cite[Eqs. 10.21.2 and 10.21.3]{DLMF}$)$ yields that $j_{\nu,1}$ is the largest first zero of any solution to \eqref{4.11a}. Hence choosing $u(x)=u_1(x)$ from \eqref{4.12a} provides the choice for $a_0$ on $(0,b)$ given in \eqref{4.9a} $($via \eqref{4.6}$)$, which yields \eqref{4.7a} with
\begin{equation}
c^{1/2}=b^{-(2+\mu-\g)/2}(2+\mu-\g)j_{\nu,1}/2,\quad \nu\in[0,\infty),\ 2+\mu-\g>0.
\end{equation}
\end{example}

\begin{example}[Power weighted distance to the boundary Bessel refinement]\label{ex4.3} 
For this example, we let $a_1(x)=d^{\g/2}_{(a,b)}(x)$ where $d_{(a,b)}(x)$ represents the distance to the boundary function defined by
\begin{equation}\label{4.22a}
d_{(a,b)}(x)=\begin{cases}
x-a,& x\in(a,(b+a)/2],\\
b-x,& x\in[(b+a)/2,b),
\end{cases}
\end{equation}
and, for $c>0$, $\nu\in[0,\infty),\ 2+\mu-\g>0$, and $x\in(a,b),\ a,b\in\bbR$, consider the Sturm--Liouville problem
\begin{align}\label{4.23a}
-\big(d^\g_{(a,b)}(x) u'(x)\big)'-\left(\frac{(1-\g)^2-(2+\mu-\g)^2\nu^2}{4}d^{\g-2}_{(a,b)}(x)+c d^\mu_{(a,b)}(x)\right)u(x)=0.
\end{align}
A solution to \eqref{4.23a} is given by the function
\begin{equation}\label{4.24a}
u_{\g,\mu,\nu}(x)=\begin{cases}
(x-a)^{\frac{1-\g}{2}}J_\nu\bigg(\dfrac{2c^{1/2} (x-a)^{\frac{2+\mu-\g}{2}}}{2+\mu-\g}\bigg),& x\in(a,(b+a)/2],\\[4mm]
A_{\g,\mu,\nu,c}(b-x)^{\frac{1-\g}{2}}J_\nu\bigg(\dfrac{2c^{1/2} (b-x)^{\frac{2+\mu-\g}{2}}}{2+\mu-\g}\bigg)\\[4mm]
\quad+B_{\g,\mu,\nu,c}(b-x)^{\frac{1-\g}{2}}Y_\nu\bigg(\dfrac{2c^{1/2} (b-x)^{\frac{2+\mu-\g}{2}}}{2+\mu-\g}\bigg),& x\in[(b+a)/2,b),
\end{cases}
\end{equation}
where
\begin{align}
\begin{split}
A_{\g,\mu,\nu,c}&=\Bigg\{1-\frac{2\pi}{\g-\mu-2}\bigg(\frac{b-a}{2}\bigg)^{\frac{\g+1}{2}} Y_\nu\bigg(\dfrac{2c^{1/2} (\frac{b-a}{2})^{\frac{2+\mu-\g}{2}}}{2+\mu-\g}\bigg)\\
&\qquad \times\Bigg[(x-a)^{\frac{1-\g}{2}}J_\nu\bigg(\dfrac{2c^{1/2} (x-a)^{\frac{2+\mu-\g}{2}}}{2+\mu-\g}\bigg)\Bigg]'\Bigg|_{x=\frac{b+a}{2}}\Bigg\},
\end{split}\\
\begin{split}
B_{\g,\mu,\nu,c}&=\frac{2\pi}{\g-\mu-2}\bigg(\frac{b-a}{2}\bigg)^{\frac{\g+1}{2}} J_\nu\bigg(\dfrac{2c^{1/2} (\frac{b-a}{2})^{\frac{2+\mu-\g}{2}}}{2+\mu-\g}\bigg)\\
&\qquad \times\Bigg[(x-a)^{\frac{1-\g}{2}}J_\nu\bigg(\dfrac{2c^{1/2} (x-a)^{\frac{2+\mu-\g}{2}}}{2+\mu-\g}\bigg)\Bigg]'\Bigg|_{x=\frac{b+a}{2}}.
\end{split}
\end{align}
The applicable choice of $c$ in this example is such that the shared derivative terms in $A_{\g,\mu,\nu,c}$ and $B_{\g,\mu,\nu,c}$ become zero, or explicitly,
\begin{equation}\label{4.32}
c^{1/2}=\bigg(\frac{b-a}{2}\bigg)^{\frac{\g-2-\mu}{2}}\dfrac{2+\mu-\g}{2}\lambda_{\g,\mu,\nu,1},
\end{equation}
where $\lambda_{\g,\mu,\nu,1}$ is the first positive zero of the function
\begin{align}\label{lambdafunc}
G_{\g,\mu,\nu}(z)=(1-\g) J_\nu(z)+(2+\mu-\g)z \frac{d}{dz}[J_\nu(z)]=(1-\g) J_\nu(z)&+\frac{2+\mu-\g}{2}z[J_{\nu-1}(z)-J_{\nu+1}(z)], \notag\\
&\nu\in[0,\infty),\ 2+\mu-\gamma>0.
\end{align}
We remark that $\lambda_{\g,\mu,\nu,1}$ is the smallest Dirichlet--Neumann-type eigenvalue of the generalized Bessel operator associated with \eqref{4.11a} on $(a,(b+a)/2)$, and refer to \cite{FS25} for more details in this direction. In particular, it was shown in \cite[Sect. 5.1]{FS25} that such constants appearing in  integral inequalities stemming from Sturm--Liouville problems with symmetric coefficient functions $($like the current example$)$ are always the Dirichlet--Neumann-type eigenvalue from the half-interval problem. We further point out that $\lambda_{0,0,0,1}$ is sometimes referred to as Lamb's constant $($see the brief discussion in \cite{AW07}$)$.

So choosing $a_0(x)=-d^{\g/2}_{(a,b)}(x)\big(u'_{\g,\mu,\nu}(x)/u_{\g,\mu,\nu}(x)\big)$ yields the integral inequality
\begin{align}\label{4.20a}
\begin{split}
\int_a^b d_{(a,b)}^\g(x) |f'(x)|^2\, dx  \geq \frac{(1-\g)^2-(2+\mu-\g)^2\nu^2}{4}\int_a^b d_{(a,b)}^{\g-2}(x)|f(x)|^2\, dx &\\
+\left(\frac{b-a}{2}\right)^{\g-\mu-2}\frac{(2+\mu-\g)^2}{4}\lambda_{\g,\mu,\nu,1}^2\int_a^b d_{(a,b)}^\mu(x) |f(x)|^2\, dx&,  \\
\nu\in[0,\infty),\ 2+\mu-\g>0,\ a,b\in\bbR,\quad f\in C^\infty_0((a,b))&,
\end{split}
\end{align}
where $\lambda_{\g,\mu,\nu,1}$ denotes the first zero of \eqref{lambdafunc}.
As before, the first coefficient on the right-hand side of \eqref{4.20a} is positive for
\begin{equation}
0\leq\nu< \frac{|1-\g|}{2+\mu-\g},\quad 2+\mu-\g>0.
\end{equation}
\end{example}

\begin{example}
As a final example motivated by the previous examples, we consider the Sturm--Liouville problem, for $\eta\in[R,\infty),\ R\in(0,\infty)$,
\begin{align}
-\big([1/\ln(\eta/x)]u'(x)\big)'-\frac{ [\ln(\eta/x)]^{2} - 2\ln(\eta/x)+3}{4 x^2[ \ln(\eta/x)]^{3}} u(x)-\frac{c u(x)}{\ln(\eta/x)}=0,\quad x\in(0,R).
\end{align}
A solution to this equation is given by
\begin{equation}
u_\eta(x)=x^{1/2}[\ln(\eta/x)]^{1/2}J_0\big(x c^{1/2}\big),
\end{equation}
so that choosing $c= j_{0,1}^2/R^2$, $a_1(x)=[1/\ln(\eta/x)]$, and $a_0(x)=-[\ln(\eta/x)]^{1/2}\big(u'_\eta(x)/u_\eta(x)\big)$ yields
\begin{align}
\begin{split}
\int_0^R \frac{|f'(x)|^2}{\ln(\eta/x)}\, dx\geq  \int_0^R \frac{ [\ln(\eta/x)]^{2} - 2\ln(\eta/x)+3}{4 x^2[ \ln(\eta/x)]^{3}} |f(x)|^2\, dx +\frac{j_{0,1}^2}{R^2}\int_0^R \frac{|f(x)|^2}{\ln(\eta/x)}\, dx,\\
\eta\in[R,\infty),\  R\in(0,\infty),\quad f\in C_0^\infty ((0,R)),
\end{split}
\end{align}
where $j_{0,1}$ is once again the first positive zero of the Bessel function $J_0(\dott)$. We point out that all of the integrands in this inequality are nonnegative under the given assumptions.
\end{example}

\subsubsection{Trigonometric refinements of Hardy's inequality and related inequalities}
\begin{example}
The choices $a_0(x)=-(1/2)\cot(x)$ and $a_1(x)=1$ on $(0,\pi)$ recover the following refinement proven in \cite{GPS21} $($with optimal constant$)$$:$
\begin{equation}\label{eq:trighardy}
\int_0^\pi |f'(x)|^2\, dx \geq \frac{1}{4}\int_0^\pi |f(x)|^2\, dx + \frac{1}{4}\int_0^\pi \frac{|f(x)|^2}{\sin^2(x)}\, dx,\quad f\in C^\infty_0((0,\pi)) .
\end{equation}
To once again motivate these choices, notice that by choosing $a_1(x)=1$ and $a_0(x)=\a\cot(x)$ on $(0,\pi)$, the requirement for nonnegativity in \eqref{5.1} becomes
\begin{equation}
[\a\cot(x)]'-\a^2\cot^2(x)=-\a\csc^2(x)-\a^2\big[\csc^2(x)-1\big]=-\big(\a^2+\a\big)\csc^2(x)+\a^2\geq0.
\end{equation}
This provides a family of refined inequalities if one has $-(\a^2+\a)\geq0$, that is, if $\a\in[-1,0]$. If one minimizes $\a^2+\a=\a(\a+1)$ by choosing $\a=-1/2$, one recovers the refinement \eqref{eq:trighardy}.
\end{example}

\begin{example}
Choosing $a_0(x)=\a\tan(x)-\beta\cot(x),$ $\a,\b\in\bbR,$ and $a_1(x)=1$ on $(0,\pi/2)$ yields, for $\a,\b\in\bbR,$ and $f\in C^\infty_0((0,\pi/2))$, the inequality
\begin{align}
\int_{0}^{\pi/2} |f'(x)|^2\, dx &\geq (\a+\b)^2\int_{0}^{\pi/2} |f(x)|^2\, dx + \big(\b-\b^2\big)\int_{0}^{\pi/2} \frac{|f(x)|^2}{\sin^2(x)}\, dx+ \big(\a-\a^2\big)\int_{0}^{\pi/2} \frac{|f(x)|^2}{\cos^2(x)}\, dx,
\end{align}
which has all nonnegative coefficients for $\a,\b\in[0,1]$ with the last two terms maximized with the choice $\a=\b=1/2$ leading to the inequality for $f\in C^\infty_0((0,\pi/2))$
\begin{equation}
\int_{0}^{\pi/2} |f'(x)|^2\, dx \geq \int_{0}^{\pi/2} |f(x)|^2\, dx + \frac{1}{4}\int_{0}^{\pi/2} \frac{|f(x)|^2}{\sin^2(x)}\, dx+ \frac{1}{4}\int_{0}^{\pi/2} \frac{|f(x)|^2}{\cos^2(x)}\, dx .
\end{equation}
This inequality, and the optimality of the constants, also stems from the spectral theoretic study of the Schr\"odinger operator with P\"oschl--Teller potential $($see, e.g., the recent study \cite{FS25a}$)$.
\end{example}

\begin{example}
Choosing $a_0(x)=\a\tanh(x)-\b\coth(x),$ $\a,\b\in\bbR,$ and $a_1(x)=1$ on $(0,\infty)$ yields for $\a,\b\in\bbR$ and $f\in C^\infty_0((0,\infty)),$
\begin{align}
\int_{0}^{\infty} |f'(x)|^2\, dx &\geq -(\a-\b)^2\int_{0}^{\infty} |f(x)|^2\, dx + \big(\a+\a^2\big)\int_{0}^{\infty} \frac{|f(x)|^2}{\cosh^2(x)}\, dx + \big(\b-\b^2\big)\int_{0}^{\infty} \frac{|f(x)|^2}{\sinh^2(x)}\, dx.
\end{align}
Hence, choosing $\a=\b$ and noting the coefficient on the third integral is maximized when $\b=1/2$ yields
\begin{equation}\label{eq:trighardylast}
\int_{0}^{\infty} |f'(x)|^2\, dx \geq \frac{3}{4}\int_{0}^{\infty} \frac{|f(x)|^2}{\cosh^2(x)}\, dx + \frac{1}{4}\int_{0}^{\infty} \frac{|f(x)|^2}{\sinh^2(x)}\, dx,\quad f\in C^\infty_0((0,\infty)).
\end{equation}
\end{example}

\begin{example}
Motivated by the power weighted Hardy's inequality \eqref{eq:weighthardy}, it is natural to consider what type of inequality can one expect when introducing a trigonometric weight to the first derivative term. In this direction, we illustrate one trigonometric weighted inequality as follows: choosing $a_1(x)=\csc(x)$ and $a_0(x)=\a\cot(x),\ \a\in\bbR,$ on $(0,\pi)$ yields
\begin{equation}\label{eq:trigweight}
\int_{0}^{\pi} \frac{|f'(x)|^2}{\sin^2(x)}\, dx \geq \int_{0}^{\pi} F_\a(x)|f(x)|^2\, dx,\quad \a\in\bbR,\ f\in C^\infty_0((0,\pi)),
\end{equation}
where
\begin{align}
F_\a(x)&=-\a\big[\a\cot^2(x)+\csc(x)+2\cot^2(x)\csc(x)\big]=-\a\big[2\csc^3(x)+\a\csc^2(x)-\csc(x)-\a\big].
\end{align}
We note that $F_\a(x)\geq0$ for $x\in(0,\pi)$ whenever
\begin{equation}\label{range}
\a\in\big[-(1/4)\big(4(5+\sqrt{17})^{1/2}+(14+2\sqrt{17})^{1/2}\big),0\big]=[-4.1995\dots,0].
\end{equation}
\end{example}

\begin{example}
A final inequality we offer that is in a similar spirit is combining a power weight and trig functions. Choosing $a_0(x)=\a\cot(x),\ a_1(x)=x^{\g/2},$ on $(0,1)$ yields, for $\a\in[-1,0],\ \g\leq0$,
\begin{align}
\int_0^1 x^\g |f'(x)|^2\, dx & \geq \a^2\int_0^1 |f(x)|^2\, dx-\a\int_0^1 \frac{\a+x^{\g/2}}{\sin^2(x)}|f(x)|^2\, dx+\frac{\a\g}{2}\int_0^1 x^{(\g/2)-1} \cot (x) |f(x)|^2\, dx,
\end{align}
where one readily verifies that each coefficient is positive under the given choices of parameters $($as one has $x^{\g/2}\geq1$ on $(0,1)$ since $\g$ is nonpositive$)$.
\end{example}

\subsection{Second-order setting}\lb{s5.2}

In the case $n=2$, from \eqref{3.2} and \eqref{diffineq} one sees that the integral inequality and associated differential inequality that would need to be satisfied for a refinement are
\begin{align}\label{5.18}
\int_a^b a_2^2(x) |f''(x)|^2\,  dx &\geq \int_a^b \big\{-a_1^2(x)+[a_1(x)a_2(x)]'+2a_0(x)a_2(x)\big\}|f'(x)|^2\, dx  \notag \\
&\quad\, + \int_a^b \big\{-a_0^2(x)+[a_0(x)a_1(x)]'-[a_0(x)a_2(x)]''\big\}|f(x)|^2\, dx,\quad f\in C^\infty_0((a,b)), \notag \\
&\hspace{-1.7cm}\begin{array}{cl}
&c_{2,1}(x)=-a_1^2(x)+[a_1(x)a_2(x)]'+2a_0(x)a_2(x)\geq0,\\
&c_{2,0}(x)=-a_0^2(x)+[a_0(x)a_1(x)]'-[a_0(x)a_2(x)]''\geq0,
\end{array}
\ \text{on}\ (a,b).
\end{align}

\begin{example}[Power weighted Rellich's inequality] 
Choosing 
\begin{align}
\begin{split}
a_0(x)&=[\a(1-\a-\g)/2 ]x^{(\g/2)-2},\ a_1(x)=\a x^{(\g/2)-1},\ a_2(x)=-x^{\g/2},\quad \g\in\bbR,\ x\in(0,\infty),
\end{split}
\end{align}
where
\begin{equation}\label{4.41a}
\a=\a_\pm=2-\g \pm \big[\big((\g-2)^2+1\big)/2\big]^{1/2},
\end{equation}
yields the power weighted Rellich's inequality $($with optimal constant$)$
\begin{equation}
\int_0^\infty x^\gamma |f''(x)|^2\, dx\geq [(\g-1) (\g - 3) /4]^2\int_0^\infty x^{\gamma-4}|f(x)|^2\, dx,\quad \g\in\bbR,\quad f\in C^\infty_0((0,\infty)).
\end{equation}
Taking $\gamma=0$ now yields the classical Rellich's inequality with constant $9/16$.

To find the appropriate choices above, one starts with
\begin{equation}
a_2(x)=-x^{\g/2},\quad a_1(x)=\a x^{(\g/2)-1},\quad a_0(x)=\b x^{(\g/2)-2},\quad \a,\b,\g\in\bbR.
\end{equation}
Then the inequalities in \eqref{5.18} reduce to the parameter inequalities for $\a,\b,\g\in\bbR$,
\begin{align}\label{a.10}
\begin{split}
-\a^2-\a(\g-1)-2\b\geq0,\quad
-\b^2+\a\b(\g-3)+\b(\g-2)(\g-3)\geq0.
\end{split}
\end{align}
Hence the underlying parameterized family of refined weighted inequalities must satisfy \eqref{a.10}. Choosing $\b=\a(1-\a-\g)/2$ results in the first inequality of \eqref{a.10} becoming equality, and the second becomes $F_\g(\a)\geq0$ where $($which coincides with the result found in \cite[Cor. 2.3]{GPPS23}$)$
\begin{equation}
F_\g(\a)=\a(1-\a-\g)[\a(\g-3)-(\a/2)(1-\a-\g)+(\g-2)(\g-3)]/2.
\end{equation}
In particular, maximizing $F_\g(\alpha)$ with respect to $\alpha$ yields maxima at \eqref{4.41a}.
\end{example}

\begin{example}[trigonometric inequality]
Choosing $a_2(x) = 1$, $a_1 = \alpha \cot(x)$, the first inequality in \eqref{5.18} becomes $2a_0 - \alpha (1+(1+\alpha)\cot^2(x))\geq0 $. Letting $\alpha = -1$ and $a_0=\b\in\mathbb{R}$ yields
\begin{align}
    \label{eq:trigrellich2}
    \int_0^\pi |f''(x)|^2\, dx \geq (2\beta +1)\int_0^\pi |f'(x)|^2\, dx + \int_0^\pi \Big(\frac{\beta}{\sin^2(x)}-\beta^2\Big)|f(x)|^2\, dx,\quad f\in C_0^\infty((0,\pi)),
\end{align}
with the function multiplying $|f|^2$ being nonnegative for $\beta\in[0,1]$ as $1/\sin^2(\pi/2)=1$.
\end{example}

\medskip

\noindent {\bf Acknowledgments.} We are indebted to Daniel Herden, Jan Lang, and Saleh Tanveer for helpful conversations regarding this project. JS was supported in part by an AMS--Simons Travel Grant.


\end{document}